\documentclass[psamsfonts]{amsart}

\usepackage{amssymb,amsfonts,amscd}

\usepackage[colorlinks,linktocpage]{hyperref}
\hypersetup{linkcolor=[rgb]{0,0,0.715}}
\hypersetup{citecolor=[rgb]{0,0.715,0}}

\newtheorem{theorem}{Theorem}[section]

\newtheorem{lemma}[theorem]{Lemma}

\newtheorem*{thm1}{Main Theorem}
\newtheorem*{thm2}{Theorem A}
\newtheorem*{thm3}{Theorem B}

\theoremstyle{definition}

\theoremstyle{remark}
\newtheorem{remark}[theorem]{Remark}

\makeatletter
\let\c@equation\c@theorem
\makeatother
\numberwithin{equation}{section}

\title[]{Ricci limit spaces are semi-locally simply connected}

\author[]{Jikang Wang}
\thanks{}

\thanks{The author is partially supported by  NSFC 11821101 and BNSF Z19003.}

\begin{document}
\date{}
\maketitle
\begin{abstract}
Let $(X,p)$ be a Ricci limit space. We show that for any $\epsilon > 0$ and $x \in X$, there exists $r< \epsilon$, depending on $\epsilon$ and $x$, so that any loop in $B_{r}(x)$ is contractible in $B_{\epsilon}(x)$. In particular, $X$ is semi-locally simply connected. Then we show that the generalized Margulis lemma holds for Ricci limit spaces of $n$-manifolds.
\end{abstract}
\section{Introduction}

 A Ricci limit space $(X,p)$ is the pointed Gromov-Hausdorff limit space of a sequence of complete $n$ dimensional Riemannian manifolds $(M_i,p_i)$ with a uniform Ricci curvature lower bound. $(X, p)$ is non-collapsing if Vol$(B_1(p_i))$ has a uniform lower bound. The regularity and geometric structure theory of $(X,p)$ have been studied extensively by Cheeger, Colding and Naber \cite{CheegerColding1997, CheegerColding2000a,CheegerColding2000b,CheegerNaber2013,CheegerNaber2015, ColdingNaber2012}. In this paper, we study the local topology of $(X,p)$.

If we further assume $M_i$ has a uniform sectional curvature lower bound, then the limit space $X$ is an Alexandrov space. In an Alexandrov sapce $(X,p)$, the tangent cone $T_p$ is a unique metric cone \cite{BGP1992} and Perelman proved that a neighborhood of $p$ is homeomorphic to the tangent cone \cite{Perelman1991}. In particular,  any Alexandrov space is locally contractible. However, it was shown in \cite{Menguy2000} that even a non-collapsing Ricci limit space may have locally infinite second Betti number and thus is not locally contractible. Due to this example, we focus on the local fundamental group of a Ricci limit space.

For a Ricci limit space $(X, p)$, Sormani and Wei showed that the universal cover of $X$ exists \cite{SormaniWei2001, SormaniWei2004}; while it was unknown the universal cover is simply connected or not. Recall that if a topological space is path-connected, locally path-connected and semi-locally simply connected, then it has a simply connected universal cover \cite{Hatcher}. $(X,p)$ is path-connected and locally path-connected. Recently Pan and Wei showed that a non-collapsing Ricci limit space is semi-locally simply connected \cite{PanWei2019}. Then Pan and the author proved in \cite{PanWang} that $(X,p)$ is semi-locally simply connected if $M_i$ has Ricci bounded covering geometry, that is,  universal covers of all $r$-balls in manifolds are non-collapsing. In the proof of \cite{PanWang}, we establish a slice theorem for pseudo-group actions; see also theorem \ref{t1}. The main theorem of this paper is that actually any Ricci limit space is semi-locally simply connected. 

\begin{thm1}
Assume $(M_i,p_i)$ is a sequence of complete $n$-manifolds with $Ric \ge -(n-1)$ and $(M_i, p_i) \overset{GH}\to (X,p)$. Then $X$ is semi-locally simply connected, i.e., for all $x \in X$, there exists $r_x$ such that any loop in $B_{r_x}(x)$ is contractible in $X$. 
\end{thm1}

The main theorem implies that $X$ has a simply connected universal cover.
Combining the main theorem  and theorem 1.4 in \cite{EnnisWei2010}, the universal cover of $X$ is also a Ricci limit space.
\begin{theorem}
Assume $(M_i,p_i)$ is a sequence of complete $n$-manifolds with $Ric \ge -(n-1)$. $(M_i,p_i)$ converges to $(X,p)$. Let $\tilde{X}$ be the universal cover of $X$. Then by passing to a subsequence, there exists $R_i \to \infty$ and cover spaces $\hat{B}(p_i,R_i)$ of the closed ball $\bar{B}_{R_i}(p_i)$ so that $(\hat{B}(p_i,R_i), \hat{p}_i)$ converges to $(\tilde{X}, \tilde{p})$.
\end{theorem}

Using the main theorem, we can easily generalize $\pi_1$-onto property in \cite{Tuschman1995} to the Ricci limit space; see also \cite{SormaniWei2001}. 
\begin{theorem}\label{onto}
Assume $M_i$ is a sequence of complete $n$-manifolds with $Ric \ge -(n-1)$ and $diam(M_i) \le D$ for some fixed $D$. Suppose $M_i$ converges to $X$. Then there exists a surjective homeomorphism $\Phi_i : \pi_1(M_i) \to \pi_1(X)$ for all large $i$. 
\end{theorem}

We will prove Theorem A below, which is stronger than the main theorem. At first we need the notion of $1$-contractibility radius; see also \cite{PanWei2019}.  
Define
\begin{equation}\nonumber
\rho (t,x)= \inf \{ \infty, r \geq t | \text{ any loop in } B_{t}(x) \text{ is contractible in } B_{r}(x) \}.  
\end{equation}

\begin{thm2}
Let $(M_{i},p_{i})$ be a sequence of $n$-manifolds converging to $(X,p)$ such that for all $i$, \\
(1). $B_{4}(p_{i}) \cap \partial M_{i}= \emptyset$ and the closure of $B_{4}(p_{i})$ is compact,\\
(2). $Ric \geq -(n-1)$ on $B_{4}(p_{i})$. \\
Then $\lim \limits_{t \to 0} \rho(t,p) = 0$.
\end{thm2}  

Kapovitch and Wilking proved the generalized Margulis lemma for $n$-manifolds with a uniform Ricci curvature lower bound \cite{KapovitchWilking2011}; see also theorem \ref{KW}. As an application of Theorem A, we will prove in section 4 that the generalized Margulis lemma holds for Ricci limit spaces of $n$-manifolds as well.

\begin{thm3}(generalized Margulis lemma) 
There exist positive constants $\epsilon$ and $C$, depending on $n$, such that the following holds. Let $(X,p)$ be the limit of a sequence of complete $n$-manifolds $(M_i,p_i)$ with $Ric \ge -(n-1)$. For any $x \in X$, the image of the natural homomorphism
$$\pi_1(B_{\epsilon}(x), x) \to \pi_1(B_1(x), x)$$
contains a nilpotent subgroup $N$ of index $\le C$. Moreover, $N$ has a nilpotent basis
of length at most $n$. 
\end{thm3}
\begin{remark}
We may have different $n$ such that $(X,p)$ is the Ricci limit space of $n$-manifolds. For example, let $X$ be a single point and $S_i$ be the circle with radius $1/i$. For each integer $n>0$, torus $S_i^n$ converges to $X$ as $i \to \infty$. In Theorem B, we may choose $n$ as the minimal integer so that $(X,p)$ is the Ricci limit space of $n$-manifolds.   
\end{remark}

On a metric space $Y$, we call two paths $\gamma:[0,1] \to Y$ and $\gamma':[0,1] \to Y$ $\epsilon$-close to each other if for any $t \in [0,1]$, $d(\gamma(t),\gamma'(t))<\epsilon$. In the case that $\gamma$ and $\gamma'$ are in different (but $\epsilon$ GH-close) spaces $Y$ and $Y'$, we have an admissible metric on the disjoint union $Y \sqcup Y'$ so that images of $Y\hookrightarrow Y \sqcup Y'$ and $Y' \hookrightarrow Y \sqcup Y'$ are $2\epsilon$ Hausdorff-close to each other. Then we define that $\gamma$ is $3\epsilon$-close to $\gamma'$  if they are $3\epsilon$-close to each other in $Y \sqcup Y'$ with the admissible metric; briefly, we say $\gamma$ is close to $\gamma'$.

Let's sketch the proof of Theorem A; the proof relies on the construction of homotopy in \cite{PanWei2019} and a slice theorem for pseudo-group actions in \cite{PanWang}.  Let $D$ be the unit disc in $\mathbb{R}^2$. Given a loop in a small ball $B_r(p)$, we want to construct a homotopy map $H_{\infty}: D \to B_{\rho}(p)$ such that $H_{\infty}(\partial D)$ is the given loop and the radius $\rho$ converges to $0$ as $r \to 0$. 

We first recall the construction of homotopy by Pan and Wei in \cite{PanWei2019}. They call $p$ Type I if there exists $r_0 > 0$ such that a family of $t$-functions 
\begin{equation*}
\{\rho (t,x_i)| x_i \in M_i, d(x_i,p_i) < r_0 \}
\end{equation*} 
are equally continuous at $t=0$. Assume $p$ is Type I. For any $x$ in $B_{r_0}(p)$ and any loop $\gamma$ in a small neighborhood of $x$, we can find $\gamma_i$ in $M_i$ which is close to $\gamma$ and contractible in a fixed ball of $x_i$ where $x_i \in M_i$ converges to $x$.
 Using such $\gamma_i$ and inductive construction, they can find a homotopy map on the limit space and show that  $\lim \limits_{t \to 0} \rho(t,p) = 0$. Notice that their proof for Type I points doesn't rely on the volume condition.
 
 However, as they mentioned, even a non-collapsing Ricci limit space may contain points which are not Type I. Therefore they had to handle with other points (Types II and III) using the non-collapsing volume condition, but their proof for Type II points can't work with a collapsing Ricci limit space. Note that a collapsing limit space may have no Type I point at all. For example, let $S_{i}$ be a circle with radius $1/i$. Then $S_{i}$ converges to a point as $i \to \infty$ and this point is not Type I.

Our key observation in this paper is that we can use slice theorem \ref{t1} to prove lemma \ref{l1}, which can play the same role of Type I condition in the construction of homotopy. Lemma \ref{l1} says that for any $x$ in a Ricci limit space and any loop $\gamma$ in a small neighborhood of $x$, we can find a loop $\gamma_i$ in $M_i$ for large $i$, which is close to $\gamma$ and has controlled homotopy property in the following sense: $\gamma_i$ is homotopic to a short loop $\gamma_i'$ and the homotopy image is contained in a fixed ball $B_{4l}(x_i)$ where $x_i$ in $M_i$ converges to $x$; the length of $\gamma_i'$ converges to $0$ as $i \to \infty$. 

In lemma \ref{l1} we get a loop $\gamma_i$ homotopic to a short loop instead of a constant loop. We will see that there is no difference between a short loop and a constant loop in the construction of homotopy on the limit space; compare lemma \ref{lemma3} in this paper with lemma 4.2 in \cite{PanWei2019}. Roughly speaking, if we see manifolds from the limit space, we can not distinguish short loops constructed in lemma \ref{l1} and constant loops in manifolds.
 
Notice that lemma \ref{l1} holds for any point in a Ricci limit space. Therefore our construction works for both collapsing and non-collapsing cases; also there is no need to classify points in the limit space as Pan and Wei did in \cite{PanWei2019}.  

To find such $\gamma_i$ and $\gamma_i'$ in lemma \ref{l1} , we consider $\bar{B}_l(\tilde{x}_i) \subset \widetilde{B_{4l}(x_i)}$ and equivariant convergence (see section 2) $$(\bar{B}_l(\tilde{x}_i),\tilde{x}_i,G_i) \to (\bar{B}_l(\tilde{x}), \tilde{x},G).$$ There is a $G_{\tilde{x}}$-slice $S$ at $\tilde{x}$ by theorem \ref{t1}. We may assume $B_r(x) \subset S/G_{\tilde{x}}$. Since $G_{\tilde{x}}$ is compact, we can lift $\gamma$ to a path $\tilde{\gamma}$ in $S$; $\tilde{\gamma}$ may not be a loop if $\gamma$ is not based on $x$. Assume two end points of $\tilde{\gamma}$ are $\tilde{z}$ and $g\tilde{z}$ where $g \in G_{\tilde{x}}$. Next we find $g_i \in G_i$ close to $g$ and $\tilde{z}_i \in \bar{B}_l(\tilde{x}_i)$ close to $\tilde{z}$. Then we can construct a path $\tilde{\gamma}_i$, in $\bar{B}_l(\tilde{x}_i)$, from $\tilde{z}_i$ to $g_i\tilde{z}_i$, so that $\tilde{\gamma_i}$ is close to $\tilde{\gamma}$. Let $\tilde{\gamma}_i'$ be a geodesic  from $\tilde{x}_i$ to $g_i\tilde{x}_i$. The length of $\tilde{\gamma}_i'$ converges to $0$ since $g \in G_{\tilde{x}}$ and $g_i$ converges to $g$ as $i \to \infty$. Now we define loops
$$\gamma_i:=\pi(\tilde{\gamma}_i), \ \gamma_i':=\pi(\tilde{\gamma}_i')$$ 
in $M_i$. $\gamma_i$ is close to $\gamma$ since $\tilde{\gamma}_i$ is close to $\tilde{\gamma}$; the length of $\gamma_i'$ is equal to the length of $\tilde{\gamma}_i'$ which converges to $0$.
$\gamma_i$ and $\gamma_i'$ are homotopic to each other in $B_{4l}(x_i)$ since both of them correspond to the deck transformation $g_i \in \pi_1(B_{4l}(x_i))$.

The homotopy between $\gamma_i$ and $\gamma_i'$ may not converge as $i \to \infty$, therefore we can't directly construct a homotopy map on the limit space by the homotopy maps on manifolds. However, by lemma \ref{lemma3}, we can use the homotopy map between $\gamma_i$ and $\gamma_i'$ to decompose $\gamma$ into many loops; each new loop is contained in a smaller ball. Then we repeat the above process for each new loop and decompose them again and again. By lemma \ref{lemma2} we will get a desired homotopy map $H_{\infty}$ which shows that $\gamma$ is contractible. Moreover, the image of $H_{\infty}$ is contained in  a fixed ball.

The author would like to thank his advisor Xiaochun Rong for many helpful discussions and Jiayin pan for a helpful comment on a revision.

\section{Preliminaries: slice of pseudo-group actions}

Let $Y$ be a completely regular topological space and $G$ be a Lie group. We call $Y$ a $G$-space if $G$  acts as homeomorphisms on $Y$. For any point $y \in Y$, define isotropy group $$G_{y}= \{ g \in G | \ gy=y \}.$$ Given a subset $S \subset Y$, we say $S$ is $G_y$-invariant if $G_y S=S$. For a $G_y$-invariant set $S$, define 
\begin{equation*}
G \times_{G_y} S = G \times S / \sim, 
\end{equation*}
with quotient topology, and $\sim$ is the equivalence relation $(g,s) \sim (gh^{-1}, hs)$ for all $g \in G, h \in G_{y}, s \in S$. There is a natural left $G$-action on $G \times_{G_y} S$ by $g:[g',s] \mapsto [gg',s]$, where $g \in G$ and $[g',s] \in G \times_{G_y} S$.

We define $S \subset Y$ a $G_y$-slice (briefly, a slice) at $y$ if the followings hold: \\
(1). $y \in S$ and $S$ is $G_y$-invariant; \\
(2). $GS$ is an open neighborhood of $y$; $[g,s] \mapsto gs$ is a $G$-homeomorphism between $G \times_{G_y} S$ and $GS$.

In particular, the second condition above implies that $(G \times_{G_y} S) / G = S/G_y$ is homeomorphic to $GS/G$.
 
The following slice theorem is due to Palais \cite{Palais1961}.
\begin{theorem}
Let $G$ be a Lie group, $Y$ be a $G$-space and $y \in Y$. The following two conditions are equivalent: \\
(1). $G_{y}$ is compact and there is a slice at $y$. \\
(2). There is a neighborhood $U$ of $y \in Y$ such that $\{g \in G | gU \cap U \neq 0 \}$ has compact closure in $G$. 
\end{theorem}

To study the local fundamental group, it is natural to consider the universal of a ball in $M_i$; see also \cite{Huang2020, HKRX2020} for related work. Let $\widetilde{B_{4}(p_i)}$ be the universal cover of $B_{4}(p_i)$ and choose $\tilde{p}_i$ such that $\pi (\tilde{p}_i)=p_i$. $\widetilde{B_{4}(p_i)}$ may have no converging subsequence \cite{SormaniWei2004}. Therefore we consider the closed ball $\bar{B}_1(\tilde{p}_i)$ in $\widetilde{B_{4}(p_i)}$, which has a converging subsequence by relative volume comparison theorem. Then let $G_{i}$ be all deck transformations $g \in \pi_1(B_{4}(p_i))$ such that $d(g\tilde{p}_i,\tilde{p}_i) \le 1/100$. Passing to a subsequence if necessary, $$(\bar{B}_1(\tilde{p}_i),\tilde{p}_i,G_{i}) \overset{eGH}\to (\bar{B}_1(\tilde{p}),\tilde{p},G).$$ 

The limit $G$ is only a pseudo-group, that is, $gg'$ is not defined for some $g,g' \in G$. Therefore we can't directly apply Palais's theorem to the limit $G$ and $\bar{B}_1(\tilde{p})$. The following existence of a slice at $\tilde{p}$ is proved in \cite{PanWang}.

\begin{theorem}\label{t1}
Given $(\bar{B}_1(\tilde{p}_i),\tilde{p}_i,G_{i}) \overset{eGH}\to (\bar{B}_1(\tilde{p}),\tilde{p},G)$ as above, there is a slice $S$ at $\tilde{p}$ so that: \\
(1). $S$ contains $\tilde{p}$ and is $G_{\tilde{p}}$-invariant; \\
(2). $S/G_{\tilde{p}}$ is homeomorphic to a neighborhood $U$ of $p$.
\end{theorem}

By theorem \ref{t1}, we can lift any path in $U$ to a path in $S$ since $G_{\tilde{p}}$ is a compact Lie group; see \cite{Bredon}.

The idea to find the slice is that we extend $G$ to a Lie group $\hat{G}$, which acts homeomorphically on an extended space. Then we apply Palais's slice theorem \cite{Palais1961} for this $\hat{G}$-space and get a slice in the extended space. Next we show that the extended group and space are locally homeomorphic to the old ones, thus we can get the slice at $\tilde{p}$ as well.

To show that extended group $\hat{G}$ is well-defined, lemma 2.1 in \cite{PanRong2018} is the key. It says for any $0 < r < R$, if an action $g$ fixes all points in $B_r(p)$, then $g$ fixes $B_R(p)$ as well. In another word,  geometric structure in small scale  can weekly control the one in large scale, therefore we can extend $G$. Lemma 2.1 in \cite{PanRong2018} and the proof of that $\hat{G}$ is a Lie group \cite{PanWang} rely on certain path connectedness of $(\mathcal{R}_k)_{\epsilon, \delta}$ \cite{CheegerColding2000a, ColdingNaber2012}.

\section{Construction of the homotopy}

We prove theorem A in this section. 
As we mentioned in section 1, for any loop $\gamma$ in a small neighborhood of $X$, we want to find a loop $\gamma_i$ in $M_i$; $\gamma_i$ is close to $\gamma$ and has the following homotopy property.
\begin{lemma}\label{l1}
Fix $x \in \bar{B}_{1/2}(p)$ and $x_i  \in B_1(p_i)$ converging to $x$. For any $l<1/2$, passing to a subsequence if necessary, there exists $r=r(x,l)<l$ and $\epsilon_i \to 0$ so that for any loop $\gamma \subset B_r(x)$, we can find loops $\gamma_i$ and $\gamma_i'$ in $M_i$ satisfying the following conditions: \\
(1). $\gamma_i$ is close to $\gamma$ and the length of $\gamma_i'$ is less than $2\epsilon_i$; \\
(2). $\gamma_i$ is homotopic to $\gamma_i'$ and the homotopy image is contained in $B_{4l}(x_i)$.
\end{lemma}
\begin{proof}
Consider $\bar{B}_{l}(\tilde{x}_i) \subset \widetilde{B_{4l}(x_i)}$ and $G_i= \{ g \in \Gamma_i | d(g\tilde{x}_i,\tilde{x}_i) \le l/100 \}$. By passing to a subsequence, we may assume 
$(\bar{B}_{l}(\tilde{x}_i), \tilde{x}_i, G_i)$ converges to $(\bar{B}_{l}(\tilde{x}),\tilde{x},G)$.
There exists a sequence of $\epsilon_i \to 0$ so that the equivariant Gromov-Hausdorff distance between $(\bar{B}_{l}(\tilde{x}_i), \tilde{x}_i, G_i)$ and $(\bar{B}_{l}(\tilde{x}),\tilde{x},G)$ is less than $\epsilon_i$. 

By theorem \ref{t1}, we can find a slice $S$ at $\tilde{x}$ such that $S/G_{\tilde{x}}$ is homeomorphic to a neighborhod $U$ of $x$. We may assume $r$ small enough such that $B_{r}(x) \subset U$. Given a loop $\gamma$ in $B_r(x)$, we construct $\gamma_i$ and $\gamma_i'$ satisfying the given conditions. 

First assume that $\gamma$ is based on $x$. Since $\gamma$ is contained in $B_{r}(x)$ and $G_{\tilde{x}}$ fixes $\tilde{x}$, $\gamma$ can be lifted to a loop $\tilde{\gamma}$ in $S$. We now construct a loop $\tilde{\gamma}_i$ based on $\tilde{x}_i$ and $\tilde{\gamma}_i$ is $5\epsilon_i$-closed to $\tilde{\gamma}$; see also \cite{PanWei2019}.

Since $\tilde{\gamma}:[0,1] \to B_l(\tilde{x})$ is equally continuous, there exists a large integer $N$ such that 
$\tilde{\gamma}([j/N,(j+1)/N])$ is contained in a $\epsilon_i$ ball
for each $0 \le j \le N-1$. For each $0 \le j \le N-1$, we can choose $q_j \in B_l(\tilde{x}_i)$ such that $d(\tilde{\gamma}(j/N),q_j) \le \epsilon_i$. Let $q_N=q_0$. Since $\tilde{\gamma}$ is a loop, $d(\tilde{\gamma}(1),q_N)= d(\tilde{\gamma}(0), q_0) \le \epsilon_i$. Then we connect $q_j$ to $q_{j+1}$ by a minimal geodesic for each $j$ and get
a loop since $q_0=q_N$. 

Re-parameterize this loop. It's direct to check $\tilde{\gamma}_i$ is $5\epsilon_i$-close to $\tilde{\gamma}$. Abusing the notation, let $\pi$ denotes quotient maps $\widetilde{B_{4l}(x_i)} \to B_{4l}(x_i)$ and $S \to S/G_{\tilde{p}}$.
 Then $\gamma_i:= \pi (\tilde{\gamma}_i)$ is $6\epsilon_i$-closed to $\gamma=\pi(\tilde{\gamma})$. Also $\gamma_i$ is contractible in $B_{4l}(x_i)$ since $\tilde{\gamma}_i$ is a loop in the  universal cover of $B_{4l}(x_i)$ . In this case, $\gamma_i'$ is a constant loop.

In general cases, we may assume that $\gamma$ is based on $z \in B_r(x)$. Lift $\gamma$ to a path $\tilde{\gamma}$ in $S$; $\tilde{\gamma}$ may not be a loop. Assume $\tilde{\gamma}$ is a path from $\tilde{z}$ to $g\tilde{z}$ where $\tilde{z} \in S$, $g \in G_{\tilde{x}}$ and $\pi(\tilde{z})=z$. We can find $\tilde{z}_i \in \bar{B}_{l}(\tilde{x}_i)$ $\epsilon$-close to $\tilde{z}$ and $g_i \in G$ $\epsilon$-close to $g$. Then $g_iz_i$ is $2\epsilon_i$-close to $gz$. By a similar construction above, we can find a path $\tilde{\gamma}_i$ from $\tilde{z}_i$ to $g_i\tilde{z}_i$, which is $10\epsilon_i$-close to $\tilde{\gamma}$. So the loop $\gamma_i:=\pi(\tilde{\gamma}_i)$ is $11\epsilon_i$-close to $\gamma$.

Let $\tilde{\gamma}_i'$ be a minimal geodesic from $\tilde{x}_i$ to $g_i\tilde{x}_i$. $d(\tilde{x}_i, g_i\tilde{x}_i) \le 2 \epsilon_i$ since $g_i$ is $\epsilon_i$-close to $g$ and $g\tilde{x}=\tilde{x}$. Therefore the length of loop $\gamma_i':=\pi(\tilde{\gamma}_i')$ is  less than $2\epsilon_i$. $\gamma_i=\pi(\tilde{\gamma}_i)$ is homotopic to the loop $\gamma_i'=\pi(\tilde{\gamma}_i')$ since both of them correspond to the deck transformation $g_i$. The homotopy image is contained in $B_{4l}(x_i)$ since we are considering the universal cover of $B_{4l}(x_i)$.
\end{proof}

The homotopy map between $\gamma_i$ and $\gamma_i'$ may not converge as $i \to \infty$. However, by lemma \ref{lemma3} below, we can use this homotopy to decompose $\gamma$ into some loops and each new loop is contained in a smaller ball; see also \cite{PanWei2019}. 
\begin{lemma}\label{lemma3}
Fix $l>0$ and $x \in \bar{B}_{1/2}(p)$, choose $r=r(x,l)$ so that lemma \ref{l1} holds. For any loop $\gamma$ in $B_{r}(x)$ and any number $j$, assuming $i$ is large enough, there exists a triangular decomposition $\Sigma$ of $D$ and a continuous map 
\begin{equation*}
H:K^1 \to B_{5l}(x),
\end{equation*}
where $K^1$ is the $1$-skeleton of $\Sigma$, so that \\
(1) $diam(\Delta)<1/j$, $diam(H(\partial \Delta)) \le 100\epsilon_i$ for any triangle $\Delta$ of $\Sigma$;\\
(2) $H |_{\partial D} = \gamma$.\\
In particular, $d(H(z), x) \le 5l$ for all $z \in K^1$. 
\end{lemma}
\begin{proof}
By lemma \ref{l1}, there exists a loop $\gamma_i$, in $M_i$, $11\epsilon_i$-close to $\gamma$; in $B_{4l}(x_i)$, $\gamma_i$ is homotopic to a short loop $\gamma_i'$; the length of $\gamma_i'$ is less than $2\epsilon_i$. 

We first construct a map $H'$ from $D$ to $B_{4l}(x_i)$. Define
\begin{equation*}
D_1 = \{ (x,y) \in \mathbb{R}^2 | 1 \ge x^{2}+y^{2} \ge 1/4 \}\ , D_2=D-D_1.
\end{equation*}
Then $S_1= \{ (x,y) \in D | x^2 + y^2=1 \}$ and $S_{1/2}=  \{ (x,y) \in D | x^2 + y^2=1/4 \}$ are boundaries of $D_1$.

Define $H' |_{S_1} = \gamma_i$ and $H' |_{S_{1/2}} = \gamma_i'$. Since  $\gamma_i$ is homotopic to $\gamma_i'$ in $B_{4l}(x_i)$, we can extend $H'$ to a continuous map from $D_1$ to $B_{4l}(x_i)$. Define $H'(D_2) =x_i$. $H'$ is continuous except at $S_{1/2}$. Now we give $D_1$ a triangular decomposition $\Sigma^1$ such that $diam(\Delta)<1/j$ and $diam(H'(\partial \Delta)) \le \epsilon_i$ for any triangle $\Delta$ of $\Sigma^1$. We can also give a triangular decomposition $\Sigma^2$ of $\bar{D}_2$, the closure of $D_2$, such that $diam(\Delta)<1/j$, $diam(H'(\partial \Delta)) \le 4\epsilon_i$ for any triangle $\Delta$ of $\Sigma^2$; this decomposition exists since $\gamma_i'$ is contained in $B_{2\epsilon_i}(x_i)$ and $H'(D_2)=x_i$. We may add some vertices in both decompositions so that $\Sigma^1$ and $\Sigma^2$ have same vertices on the circle $S_{1/2}$. Then their union $\Sigma= \Sigma_1 \cup \Sigma_2$ is a triangular decomposition of $D$ such that $diam(\Delta)<1/j$ and $diam(H'(\partial \Delta)) \le 4\epsilon_i$ for any triangle $\Delta$ of $\Sigma$.

Let $K^1$ be the $1$-skeleton of $\Sigma$. Now we construct a continuous map $H:K^1 \to B_{5l}(x)$; see also \cite{PanWei2019}. Let $K^0$ be all vertices on $\Sigma$. If $v \in K^0$ is on $\partial D$, $H(v)$ is defined by $H |_{\partial D}= \gamma$ (in particular, $d(H(v),H'(v)) \le 11 \epsilon_i$); otherwise define $H(v)$ be a point in $\bar{B}_l(x)$ such that
$d(H(v),H'(v)) < 2\epsilon_i$.

For any two points $v$ and $u$ in $K^0$ connected by an edge $e$, if $e$ is a part of the $\partial D$, then $H$ is defined on $e$ since $H |_{\partial D}= \gamma$; otherwise we define $H |_{e}$ be the minimal geodesic between $H(v)$ and $H(u)$. Since we always have
$$ d(H(v),H(u)) \le d(H(v),H'(v)) + d(H'(v),H'(u)) + d(H'(u),H(u)) \le 30 \epsilon_i,$$ 
the image of $H |_e$ is contained in $B_{30\epsilon_i}(H(v))$. Therefore $diam(H(\partial \Delta)) \le 100\epsilon_i$ for any triangle $\Delta$ of $\Sigma$. 

Finally we show  $H(K^1) \subset B_{5l}(x)$. By lemma \ref{l1} and the construction of $H'$, $H'(D) \subset B_{4l}(x_i)$. Since $d(H(v),H'(v)) < 11\epsilon_i$ for any $v \in K^0$ and $diam(H(\partial \Delta)) \le 100\epsilon_i$ for any triangle $\Delta$, the image of $H$ is contained in $B_{4l+120\epsilon_i}(x)$. We may assume $i$ is large enough so that $120\epsilon_i < l$, then $H(K^1)$ is contained in $B_{5l}(x)$.
\end{proof}

We will use the following lemma in \cite{PanWei2019}.
\begin{lemma}\label{lemma2}
For any $j \ge 10$, let $\Sigma_j$ be a finite triangular decompositions of unit disc $D$ with the conditions below: \\
(1). each $\Sigma_{j+1}$ is a refinement of $\Sigma_j$; \\
(2). $diam(\Delta) \le 1/j$ for every triangle $\Delta$ of $\Sigma_j$. \\
Suppose that we have a sequence of continuous maps $H_j : K^1_j \to B_{1/3}(p)$, where $K^1_j$ is the $1$-skeleton of $\Sigma_j$, such that for all $j \ge 10$, \\
(3). $H_{j+1}|_{K^1_j} = H_j$; \\
(4). for any $z \in K^1_{j+1}-K^1_j$, $d(H_{j+1}(z),H_j(u)) \le  1/2^j$ holds for all $u$ in the boundary of $\Delta$, where $\Delta$ is a triangle of $\Sigma_k$ containing $z$. \\
Then $H_j$ converges to a continuous map $H_{\infty} : D \to \bar{B}_{1/3}(p)$. 
\end{lemma}


Now we can prove  Theorem A. The basic idea is to decompose a given loop into many loops and each new loop is contained in a smaller ball (lemma \ref{l1} and \ref{lemma3}). Then repeat the above process to decompose new loops. By induction, we get a homotopy map by lemma \ref{lemma2}.

\begin{proof}
Let $l_j=1/2^j$. For a fixed $j \ge 10$, although $r(x,l_j)$ in lemma \ref{l1} depends on the choice of base point $x$, we can find a finite set $S(l_j) \subset \bar{B}_{1/2}(p)$ so that $\{ B_{r(x,l_j)}(x)| x \in S(l_j) \}$ covers $\bar{B}_{1/2}(p)$.  Therefore there exists $r_j< 1/2^j$ such that for any $z \in \bar{B}_{1/3}(p)$, $B_{r_j}(z)$ is contained in one of $B_{r(x,l_j)}(x)$ where $x \in S(l_j)$.    

Let's start with a fixed integer $J \ge 10$. We will show that any loop  $\gamma$ in $B_{r_J}(p)$ is contractible in $B_{2^{-J+4}}(p)$. 

Since $B_{r_J}(p) \subset  B_{r(x,l_J)}(x)$ for some $x \in S(l_J)$, $\gamma$ is contained in $B_{r(x,l_J)}(x)$. Use lemma \ref{l1} and \ref{lemma3} for $\gamma$ and $B_{r(x,l_J)}(x)$. We may  choose large $i$ such that $\epsilon_i<r_{J+1}/100$. Then we can find a triangular decomposition $\Sigma_1$ of $D$ and a continuous map on the $1$-skeleton $K^1_1$,
\begin{equation*}
H_1:K_1^1 \to \bar{B}_{5l_J}(x)
\end{equation*} 
with the following properties: \\
(1A) $diam(\Delta)<1/(J+1)$ and $diam(H_1(\partial \Delta)) \le 100\epsilon_i < r_{J+1}$ for any triangle $\Delta$ of $\Sigma_1$;\\
(1B) $H_1(\partial D)= \gamma$, $d(H_1(z), q)
 \le d(H_1(z), x) + d(x,q) \le 5l_J + r(x,l_J) \le 6 \times 2^{-J} \le 2^{-J+3}$ for any $z \in K_1^1$ and $q$ in $\gamma$. \\

Now we consider the same procedure for loop $H_1(\partial \Delta)$ where $\Delta$ is a triangle in $\Sigma_1$.  Since $diam(H_1(\partial \Delta)) \le r_{J+1}$ by (1A),  $H_1(\partial \Delta)$ is contained in $B_{r(x',l_{J+1})}(x')$ for some $x' \in S(l_{J+1})$. Assume $x_i' \in B_1(p_i)$ converges to $x'$. Use lemma \ref{l1} and lemma \ref{lemma3} for $H_1(\partial \Delta)$ and choose $i$ large enough such that $\epsilon_i < r_{J+2}/100$ , we can find a triangular decomposition $\Sigma_{2,\Delta}$ of $\Delta$ and a continuous map on $1$-skeleton $K^1_{2,\Delta}$, 
\begin{equation*}
H_{2,\Delta}:K^1_{2,\Delta} \to \bar{B}_{5l_{J+1}}(x')
\end{equation*}
such that \\
(2A) $diam(\Delta')<1/(J+2)$, $diam(H_{2,\Delta}(\partial \Delta')) \le 100\epsilon_i < r_{J+2}$ for any triangle $\Delta'$ of $\Sigma_{2,\Delta}$; \\
(2B) $H_{2,\Delta} |_{\partial \Delta} = H_1 |_{\partial \Delta}$, and 
\begin{equation*}
d(H_{2,\Delta}(z), H_1(u)) \le d(H_{2,\Delta}(z),x') + d(x', H_1(u)) \le 5l_{J+1} + r(x',l_{J+1}) \le 2^{-J+2}
\end{equation*}
for all $z \in K^1_{2,\Delta}$ and $u \in \partial \Delta$. Do this for any triangle $\Delta$ of $\Sigma_1$, we get $\Sigma_2$, a triangular decomposition of $D$ which refines $\Sigma_1$. We also get a continuous map $H_2: K_2^1 \to B_{1}(p)$ such that 
\begin{equation*}
d(H_{2}(z), H_1(u)) \le 2^{-J+2}
\end{equation*} 
for all $z \in K^1_{2}-K^1_1$ and all $u \in \partial \Delta$ where $\Delta$ is a triangle of $\Sigma_1$ containing $z$.

Repeat the above process, we can find a sequence of triangular decomposition $\Sigma_j$ and continuous maps $H_k$ on $K_k^1$ such that  \\
(kA) $diam(\Delta)<1/(J+k)$, $diam(H_k(\Delta)) < r_{J+k}$, for any triangle $\Delta$ of $\Sigma_k$; $\Sigma_{k}$ is a refinement of $\Sigma_{k-1}$; \\
(kB) $H_{k} |_{K_{k-1}^1} = H_{k-1}$, and 
\begin{equation*}
d(H_{k}(z), H_{k-1}(u)) \le 2^{-J-k+4}
\end{equation*}
for all $z \in K^1_{k}$ and all $u \in \partial \Delta$ where $\Delta$ is a triangle of $\Sigma_{k-1}$ containing $z$. 

To continue the above process, we must show that the image of $H_k$ can't leave $B_{1/3}(p)$ where we can apply lemma \ref{l1}. However, by (kB) we have 
\begin{equation}\label{e1}
d(H_{k}(z),p) \le \sum_{j=1}^{k} 2^{-J-j+4}< 2^{-J+4}
\end{equation}  
for all $z \in K^1_{k}$.
Then the image of $H_{k}$ is always contained in $B_{2^{-J+4}}(p)$. 

By (kA) and (kB), $H_k$ satisfies conditions in lemma \ref{lemma2}. Therefore  $H_k$ converges to a continuous map $H_{\infty}: D \to B_{1/3}(p)$. Actually the image of $H_{\infty}$ is contained in $B_{2^{-J+4}}(p)$ due to the inequality (\ref{e1}).
In particular, $\gamma$ is contractible in $B_{2^{-J+4}}(p)$ by $H_{\infty}$. Therefore $\rho(r_J,p) \le 2^{-J+4}$. Let $J \to \infty$ we get $\lim_{t \to 0} \rho(t,p) =0$. 
\end{proof}

\section{Generalized Margulis lemma in Ricci limit spaces}

Let's recall the main theorem in \cite{KapovitchWilking2011}.

\begin{theorem}\label{KW}
In each dimension $n$ there are positive constants $C(n)$ and $\epsilon(n)$ such that the following holds for any complete n dimensional Riemannian manifold $(M, g)$ with $Ric \ge -(n-1)$ on a metric ball
$B_1(p) \subset M$. The image of the natural homomorphism
$$ \pi_1(B_{3\epsilon}(p), p) \to \pi_1(B_1(p), p)$$
 contains a nilpotent subgroup $N$ of index $\le C$. Moreover, $N$ has a nilpotent basis
of length at most $n$. 
\end{theorem}

Now we fix $n$. There exist $C$ and $\epsilon$  such that theorem \ref{KW} holds. We will show Theorem B holds for this pair of $C$ and $\epsilon$. Without losing of generality, we may choose $x=p$ in Theorem B. 

Let $(M_{i},p_{i})$ be a sequence of complete $n$-manifolds converging to $(X,p)$ and $Ric \ge -(n-1)$ in $M_i$.
let $I$ be the image of the natural homomorphism
$$ \pi_1(B_{\epsilon}(p), p) \to \pi_1(B_1(p), p).$$
To prove Theorem B, we only need to show $I$ contains a nilpotent subgroup $N$ of index less than $C$.

Let $I_i$ be the image of  
$$\pi_1(B_{2\epsilon}(p_i), p_i) \to \pi_1(B_{1-\epsilon}(p_i), p_i),$$
$I'$ be the image of $$\pi_1(B_{3\epsilon}(p), p) \to \pi_1(B_1(p), p).$$
By theorem \ref{KW} and a rescaling argument, $I_i$ has a nilpotent subgroup $N_i$ of index less than $C$. Then we only need the following lemma \ref{l4}, which is a local version of $\pi_1$-onto property. Although the proof of lemma \ref{l4} can be found in \cite{PanWei2019, Tuschman1995}, we give it here for reader's convenience.

\begin{lemma} \label{l4}
When $i$ is large enough, there exists a group homomorphism $$\Phi_i:I_i \to I'$$ and the image of $\Phi_i$ contains $I$.
\end{lemma}

\begin{proof}
We claim there exists a small number $\delta$ such that for all $x \in B_{1-\epsilon}(p)$, any loop in $B_{\delta}(x)$ is contractible in $B_1(p)$. Assume there is no such $\delta$. For any $j$, there exists $x_j \in B_{1-\epsilon}(p)$ such that some loops in $B_{1/j}(x_j)$ are not contractible in $B_1(p)$. Passing to a subsequence, we may assume $x_j$ converges to $x \in \bar{B}_{1-\epsilon}(p)$. For any $r>0$, $B_r(x)$ contains $B_{1/j}(x_j)$ when $j$ is large enough. Therefore some loops in $B_r(x)$ are not contractible in $B_1(p)$, which is contradictory to Theorem A.

We may choose $i$ large enough so that $d_{GH}((B_2(p),p), (B_2(p_i),p_i)) \le \delta/200$. For any $[\gamma_i] \in I_i$ where $\gamma_i$ is a loop in $B_{2\epsilon}(p_i)$ with base point $p_i$, we can find a loop $\gamma$ in $B_{3\epsilon}(p)$ such that $\gamma$ is $\delta/10$-close to $\gamma_i$ and $\gamma$ is based on $p$. Define $\Phi_i([\gamma_i])=[\gamma] \in I'$.

We check that $\Phi_i$ is well-defined. First we show that $\Phi_i([\gamma_i])$ does not depend on the choice of $\gamma$. Assume $\gamma'$ is another loop in $B_{3\epsilon}(p)$, based on $p$, such that $\gamma'$ is $\delta/10$-close to $\gamma_i$. We will show that $\gamma$ is homotopic to $\gamma'$ in $B_1(p)$. Since both $\gamma$ and $\gamma'$ are $\delta/10$-close to $\gamma_i$, $\gamma$ is $\delta/5$-close to $\gamma'$. Let $0=t_1 < t_2 ... < t_J = 1$ be a division of $[0,1]$ so that $d(\gamma(t_j),\gamma(t_{j+1}) < \delta/5$ for all $ 0 \le j \le J-1$. Define $H(0,t)=\gamma(t)$, $H(1,t)=\gamma'(t)$ for $t \in [0,1]$. Fix $t_j$, let $H(s,t_j)$ be a geodesic from $\gamma(t_j)$ to $\gamma'(t_j)$. let $R_j \subset [0,1] \times [0,1]$ be the rectangle with vertexes $(0,t_j), (0,t_{j+1}), (1, t_j), (1, t_j)$. Then $H(\partial R_j)$ is contained in $B_{\delta}(H(0,t_j))$. For now $H$ is only defined on the boundary of each $R_i$. By the choice of $\delta$, we can extend $H: [0,1] \times [0,1] \to B_1(p) $ continuously. Since $H(0,t)=\gamma(t)$ and $H(1,t)=\gamma'(t)$, $\gamma$ is homotopic to $\gamma'$ in $B_1(p)$.

Then we show that $\Phi_i([\gamma_i])$ does not depend on the choice of $\gamma_i$, therefore $\Phi_i$ is well-defined. Assume $\gamma_i$ and $\gamma_i'$ are two loops in $B_{2\epsilon}(p_i)$, which are  homotopic to each other in $B_{1-\epsilon}(p_i)$. In $B_{3\epsilon}(p)$, we can find $\gamma$ $\delta/5$-close to $\gamma_i$, $\gamma'$ $\delta/5$-close to $\gamma_i'$. We will show $\gamma$ is homotopic to $\gamma'$ in $B_1(p)$. Since $\gamma_i$ is homotopic to $\gamma_i'$, by the same construction of lemma \ref{lemma3} we can get a triangular decomposition of $[0,1] \times [0,1]$ and a continuous map $H$ from the $1$-skeleton of this decomposition to $B_{1}(p)$; $H(0,t)=c(t)$ and $H(1,t)=c'(t)$ for all $t \in [0,1]$. Moreover, for each triangle $\Delta$, $H(\partial \Delta)$ is contained in a $\delta$ ball  $B_{\delta}(x)$ where $x \in B_{1-\epsilon}(p)$. By the choice of $\delta$, we can extend $H$ and get a homotopy map between $\gamma$ and $\gamma'$. The image of $H$ is contained in $B_1(p)$. Therefore $\Phi_i([\gamma_i])$ is independent of the choice of $\gamma_i$.

$\Phi_i$ is a group homomorphism by the definition. We show that the image of $\Phi_i$ contains $I$. Choose any loop $\gamma$ in $B_{\epsilon}(p)$. We can find a loop $\gamma_i$ in $B_{2\epsilon}(p_i)$ $\delta/5$-close to $\gamma$ when $i$ is large enough. Then $\Phi_i([\gamma_i])=[\gamma]$.
\end{proof}

Now we can prove Theorem B:
\begin{proof}
Fix a large $i$. By theorem \ref{KW} and a rescaling argument, $I_i$ contains a nilpotent subgroup $N_i$ of index $\le C$. $N_i$ has a nilpotent basis
of length at most $n$. 

Let $\Phi_i$ be the homomorphism in lemma \ref{l4}.  Define $N:=\Phi_i(N_i) \cap I$. Since $\Phi_i$ is a group homomorphism, $N$ is nilpotent subgroup of $I$ and $N$ has a nilpotent basis
of length at most $n$. 

We show that the index of $N$ in $I$ is less than $C$. Assume there exists $C+1$ elements $g_j \in I$, of which the image $[g_j] \in I/N$ are different to each other, $1 \le j \le C+1$. Since $\Phi_i(I_i)$ contains $I$, we can find $g_j' \in I_i$ such that $\Phi_i(g_j')=g_j$. We claim $[g_j'] \in I_i/N_i$ are different to each other as well, which leads to a contradiction.

Now we prove the above claim. Assume $[g_1']=[g_2']$. Then $g_2'=g_1'g$ for some $g \in N_i$. $\Phi_i(g)=\Phi_i(g_2') \Phi_i(g_1')^{-1}= g_2g_1^{-1} \in I$. So $\Phi_i(g) \in \Phi_i(N_i) \cap I = N$. Since $g_2=g_1 \Phi_i(g)$, $[g_1]=[g_2] \in I/N$. It is contradictory to the choice of $g_j$. 
\end{proof}

\bibliographystyle{plain} 
\bibliography{bib}
\end{document}